\documentclass[11pt,reqno]{amsart}
\usepackage[T1]{fontenc}
\usepackage[utf8]{inputenc} 
\usepackage{lmodern}
\usepackage[english]{babel}
\usepackage{microtype}

\usepackage[%
    a4paper,
	left=2.5cm,       
	right=2.5cm,      
	top=3.5cm,        
	bottom=3.5cm,     
	heightrounded,    
	bindingoffset=0mm 
]{geometry}

\usepackage{graphicx} 
\usepackage{float} 

 \usepackage[parfill]{parskip} 
 
\usepackage{booktabs} 
\usepackage{array} 
\usepackage{paralist} 
\usepackage{verbatim} 
\usepackage{subfig} 
\usepackage{mathrsfs}
\usepackage{amssymb}
\usepackage{xcolor}
\usepackage{amsthm}
\usepackage{amsmath,amsfonts,amssymb,esint,hyperref}
\usepackage[noabbrev, capitalize]{cleveref}

\usepackage{autonum}

\usepackage{graphics,color}
\usepackage{enumerate, enumitem}
\usepackage{mathtools,centernot}
\usepackage{cases}
\usepackage{amsrefs}
\usepackage{bbm}
\usepackage{xfrac}

\usepackage{bookmark}

\newtheorem{theorem}{Theorem}[section]
\newtheorem{lemma}[theorem]{Lemma}

\newtheorem{definition}[theorem]{Definition}
\newtheorem{proposition}[theorem]{Proposition}
\newtheorem{remark}[theorem]{Remark}

\numberwithin{equation}{section} 

\newcommand{\norm}[1]{\left\|#1\right\|}

\newcommand{\T}{\ensuremath{\mathbb{T}}}
\newcommand*{\R}{\ensuremath{\mathbb{R}}}

\newcommand*{\Z}{\ensuremath{\mathbb{Z}}}

\newcommand{\eps}{\varepsilon}

\newcommand{\hamx}{\dot{H}^{-\frac12}_x}
\newcommand{\ham}{\dot{H}^{-\frac12}}

\newcommand{\quotes}[1]{``#1''}

\newcommand{\mez}{\frac{1}{2}}

\renewcommand{\MR}[1]{} 

\usepackage{color, graphicx}
\usepackage{mathrsfs, dsfont}

\usepackage[]{hyperref}
\hypersetup{
    colorlinks=true,       
    linkcolor=red,          
    citecolor=blue,        
    filecolor=red,      
    urlcolor=cyan           
}

\newcommand{\be}{\begin{equation}}
\newcommand{\ee}{\end{equation}}

\title{Global existence, Hamiltonian conservation and vanishing viscosity for the surface quasi-geostrophic equation}

\author[L. De Rosa]{Luigi De Rosa}
\address[L. De Rosa]{Gran Sasso Science Institute, viale Francesco Crispi, 7, 67100 L’Aquila, Italy}
\email{luigi.derosa@gssi.it}

\author[M. Latocca]{Micka\"el latocca}
\address[M. Latocca]{Laboratoire de Mathématiques et de Modélisation d’Évry (LaMME), Université d’Évry,
23 Bd François Mitterrand, 91000 Évry-Courcouronnes, France}
\email{mickael.latocca@univ-evry.fr}

\author[J. Park]{Jaemin Park}
\address[J. Park ]{Departement of Mathematics,
Yonsei university
50 Yonsei-Ro, Seodaemun-Gu, 03722 Seoul, South Korea}
\email{jpark776@yonsei.ac.kr}

\date{\today}

\thanks{\textit{Acknowledgments.} JP was partially supported by SNSF Ambizione
fellowship project PZ00P2-216083, the Yonsei University Research Fund of 2024-22-0500, the
POSCO Science Fellowship of POSCO TJ Park Foundation, HYUNSONG Educational and Cultural Foundation, and  the National Research Foundation of Korea (NRF) grant funded by the Korea government (MEST) No. RS-2026-25476891.}

\begin{document}

\begin{abstract}
For any initial datum $\theta_0\in L^{\frac{4}{3}}_x$ it is proved the existence of a global-in-time weak solution $\theta \in L^\infty_t L^{\frac43}_x$ to the surface quasi-geostrophic equation whose Hamiltonian, i.e. the $\dot{H}^{-\frac{1}{2}}_x$ norm, is constant in time. The solution is obtained as a vanishing viscosity limit.  The main idea is to propagate in time the non-concentration of the $L^{\frac{4}{3}}_x$ norm of the initial data, from which the strong compactness in the Hamiltonian norm is deduced. Minimal Onsager supercritical conditions preventing anomalous dissipation are given.
\end{abstract}

\maketitle

\section{Introduction}
In the two-dimensional spatially periodic setting, we consider the surface quasi-geostrophic equation
\begin{equation}\label{SQG}\tag{SQG}
  \begin{array}{rcll}
  \partial_t\theta + u\cdot\nabla \theta &=& 0 & \\ 
  u &=& \mathcal R^\perp \theta &\\ 
  \theta(0,\cdot)&=&\theta_0,&
  \end{array}
\end{equation}  
where $\mathcal R^\perp := \nabla^\perp (-\Delta)^{-\mez}$ is the orthogonal vector-valued Riesz transform, $\theta : [0,\infty) \times \mathbb{T}^2 \to \mathbb{R}$ and $\theta_0:\T^2\rightarrow \R$ is a given initial datum. For convenience in the exposition, we will restrict to the case\footnote{Whenever the initial datum is only a periodic distribution, the zeros average condition should be interpreted as $\langle \theta_0,1\rangle =0$. This will not be specified anymore.}
$$
\int_{\T^2}\theta_0(x)\,dx=0,
$$
and note that all the results obtained in the current paper generalize to non-zero average initial data with minor modifications. The zero average condition is indeed propagated in time by \eqref{SQG} and it allows for Sobolev norms of homogeneous type.

Note that \eqref{SQG} formally conserves the Hamiltonian
\begin{equation}
    \label{eq.hamiltonian}
    \mathcal H(t):=\|\theta(t)\|_{\dot{H}^{-\mez}_x}^2, 
\end{equation}
which can be seen by taking the $\dot{H}^{-\mez}(\T^2)$ scalar product between the first equation in \eqref{SQG} and $\theta$. Here $\dot{H}^{-\mez}(\T^2)$ denotes the homogeneous Sobolev space (see Section \ref{S:tools}). Similarly, because $u$ is divergence-free and $\theta$ is transported by $u$, there also holds 
\[
    \|\theta(t)\|_{L_x^p} = \|\theta_0\|_{L_x^p} \qquad \forall t\geq 0, \, \forall p\in [1,\infty].
\]

We also consider the critical\footnote{Here  \quotes{critical} refers to the dissipative term $(-\Delta)^\frac12$.} viscous counterpart of \eqref{SQG} which writes
\begin{equation}\label{viscous-SQG}\tag{SQG$_{\nu}$}
  \begin{array}{rcll}
  \partial_t\theta^{\nu} + u^{\nu}\cdot\nabla\theta^{\nu} + \nu (-\Delta)^{\frac12}\theta^{\nu}&=& 0 &\\ 
  u^{\nu} &=& \mathcal R^\perp \theta^{\nu}& \\ 
  \theta^{\nu}(0,\cdot )&=&\theta^{\nu}_0.&
  \end{array}
\end{equation}
Solutions to \eqref{viscous-SQG}, if smooth enough, enjoy the energy balance
\begin{equation}
    \label{eq.viscous-conservation}
    \|\theta^{\nu}(t)\|^2_{\dot{H}_x^{-\mez}} + 2\nu \int_0^t\|\theta^{\nu}(\tau)\|^2_{L^2_x} d\tau = \|\theta^{\nu}_0\|^2_{\dot{H}_x^{-\mez}} \qquad \forall t\geq 0.
\end{equation}

\subsection{Main results}

Our main result is the existence of global weak solutions to \eqref{SQG} from every initial datum with critical\footnote{Here \quotes{critical} refers to the fact that $p=\frac43$ is the smallest value such that $L^p(\T^2)\subset \dot H^{-\frac12}(\T^2)$.} integrability. We will denote by $C^0_{\rm w}([0,\infty);L^p(\T^2))$ the space of continuous-in-time functions with values in $L^p(\T^2)$ endowed with the weak topology.

\begin{theorem}\label{T:global-existence}
Let $\theta_0\in L^\frac43 (\T^2)$ be with zero average. There exists a global-in-time weak solution 
$$
\theta\in C^0([0,\infty);\dot{H}^{-\frac12}(\T^2))\cap C^0_{\rm w}([0,\infty);L^\frac43(\T^2)) 
$$
to \eqref{SQG} in the sense of Definition \ref{def.sqg-solutions} such that 
$$
\int_{\T^2}\theta(t,x)\,dx=0,\quad \|\theta (t)\|_{\hamx}=\|\theta_0\|_{\hamx} \quad\text{and}\quad \|\theta (t)\|_{L^\frac43_x}\leq \|\theta_0\|_{L^\frac43_x}
$$
for all times $t\geq 0$.
\end{theorem}

Theorem \ref{T:global-existence} is obtained as a consequence of the following no anomalous dissipation and strong compactness result for vanishing viscosity solutions to \eqref{viscous-SQG}. Note that \eqref{viscous-SQG} has a unique global-in-time smooth solution from every smooth initial datum \cite{KNV07}.

\begin{theorem}\label{T:main vanish visc}
Let $\{\theta^\nu\}_{\nu}$ be the sequence of smooth solutions to \eqref{viscous-SQG}  with zero average initial data $\{\theta_0^{\nu}\}_\nu \subset C^\infty (\T^2)$ such that $\{|\theta_0^\nu|^\frac43\}_\nu\subset L^1(\T^2)$ is weakly compact. There exist $\theta_0\in L^\frac43(\T^2)$ and $\theta\in C^0_{\rm w}([0,\infty);L^\frac43(\T^2))$ with $\theta(0)=\theta_0$ such that, up to subsequences, 
\begin{itemize}
    \item[(a)]  $\theta^\nu\overset{*}{\rightharpoonup} \theta$ in $L^\infty([0,\infty);L^\frac43(\T^2))$  and  $\theta^\nu(t)\rightarrow \theta(t)$ in $\ham (\T^2)$ for all $t\geq 0$;
    \item[(b)]  $\theta$ is a weak solution to \eqref{SQG} with initial datum $\theta_0$ in the sense of Definition \ref{def.sqg-solutions};
    \item[(c)]  there is no anomalous dissipation of the Hamiltonian, that is 
    $$
    \lim_{\nu\rightarrow 0}\nu \int_0^T \|\theta^\nu(\tau)\|^2_{L^2_x}\,d\tau=0\qquad \forall T<\infty;
    $$
    \item[(d)] it holds $\|\theta(t)\|_{\hamx}=\|\theta_0\|_{\hamx}$ and $\|\theta (t)\|_{L^\frac43_x}\leq \liminf_{\nu\rightarrow 0} \|\theta^\nu_0\|_{L^\frac43_x}$ for all $t\geq 0$. In particular $\theta\in C^0([0,\infty);\dot H^{-\frac12}(\T^2))$.
\end{itemize}
The strong compactness in $(a)$ is in fact uniform in time (see Remark \ref{R:uniform in time conv}). Moreover, see Theorem \ref{T:compact implies no AD}, the only strong $L^2([0,T];\dot H^{-\frac12}(\T^2))$ compactness of $\{\theta^\nu\}_\nu$, together with the one of $\{\theta^\nu_0\}_\nu$ in $\dot H^{-\frac12}(\T^2)$, is enough to rule out the dissipation of the Hamiltonian in the inviscid limit. In fact, compactness at frequencies $\sim \nu^{-1}$ suffices (see Theorem \ref{T: comp diss scale}).

\begin{remark}
    Theorem \ref{T:main vanish visc} would still hold without the smoothness assumption on the initial data and for any sequence of (appropriately defined) Leray solutions, which are not known to be unique for general $L^\frac{4}{3}(\T^2)$ initial data. However, restricting to smooth initial data is enough for our purposes.
\end{remark}
\end{theorem}

\subsection{Related literature and our new contributions} 

The inviscid surface quasi-geostrophic \eqref{SQG} equation and its critical viscous counterpart \eqref{viscous-SQG} have received considerable attention in the mathematical fluid dynamics community. 

The analysis of the two-dimensional inviscid \eqref{SQG} equation was initiated in \cite{ConstantinMajdaTabak1994}, highlighting the  strong analogies between $\nabla^{\perp}\theta$  and the vorticity of the three-dimensional Euler equations. 
As in the Euler case, the possibility of finite-time blow-up for smooth solutions remains a major open problem.
The existence of non-trivial global smooth solutions is itself a significant challenge, see for instance \cite{CastroCordobaGomezSerrano2016}. 
Local-in-time solutions have been constructed in various regularity classes such as $H^s(\mathbb{R}^2)$ for $s>2$ and $C^{1,\alpha}(\mathbb{R}^2)$ for $\alpha >0$, see for instance \cite{ConstantinMajdaTabak1994} in which  a Beale--Kato--Majda blowup criterion has also been derived. We refer to \cite{JeongKimMiura2025} for a recent extension to the half-space setting and to \cites{ConstantinNguyen2018a,ConstantinNguyen2018b} for the study of \eqref{SQG} on general bounded domains. Further interesting questions related to the inviscid \eqref{SQG} equation are treated in \cites{GravejatSmets2017,CastroCordobaGomezSerrano2016,HeKiselev2021}.

Due to the presence of (at least) two formally conserved quantities in \eqref{SQG} at different levels of regularity, namely the $\dot{H}^{-\frac{1}{2}}(\T^2)$ norm, i.e. the Hamiltonian, and any $L^p(\T^2)$ norm, it is natural to expect that turbulent solutions may exhibit a dual energy cascade, as predicted by the Batchelor--Kraichnan theory (see \cites{Constantin1998,C02,ConstantinTarfuleaVicol2014} for more details).
This perspective has motivated a substantial work focused on weak solutions to \eqref{SQG}. For initial data $\theta_0\in L^2(\R^2)$, global-in-time weak solutions in the class $L^{\infty}([0,\infty);L^2(\R^2))$ were first constructed in \cite{Res95} by Resnick. This has been later extended by Marchand in \cite{Mar08a} to the case $\theta_0\in L^p(\R^2)$, for $p>\frac43$,  producing solutions in $L^{\infty}([0,\infty);L^p(\mathbb{R}^2))$. To the best of our knowledge it was not known whether the latter solutions conserve the Hamiltonian, an issue that we handle in the current work.

The viscous system \eqref{viscous-SQG} arises in geophysical studies of strongly rotating fluid flows (see for instance \cite{C02} and references therein). For \eqref{viscous-SQG}, parabolic regularity techniques \`a la De Giorgi, initiated in this context in \cite{CaffarelliVasseur2010}, show that any appropriately defined Leray solution is in fact smooth for all positive times. Alternative proofs of global regularity for \eqref{viscous-SQG} for smooth initial data have been given in \cites{KNV07,KiselevNazarov2010}. We refer  to \cites{Marchand2005,Marchand2006,LazarXue2019} for further results in this direction. 
In the supercritical dissipative case, i.e. when the dissipation in \eqref{viscous-SQG} is given by $(-\Delta)^\alpha$ for $\alpha<\frac12$, it has been established \cite{ConstantinWu2008} that any Leray weak solution which also belongs to $C^{\delta}(\mathbb{R}^2)$ for some $\delta > 1 - 2 \alpha$ on a time-interval $[t_0,T]$ is in fact smooth for all $t \in (t_0,T]$.

The main goal of our work is to construct global weak solutions $\theta \in L^\infty([0,\infty);L^{\frac{4}{3}}(\T^2))$ to \eqref{SQG} that conserve the Hamiltonian. In particular, our Theorem \ref{T:global-existence} extends the works \cites{Res95,Mar08a} mentioned above to the critical integrability $\theta_0\in L^\frac43(\T^2)$, providing solutions which, in addition, conserve the Hamiltonian for all times. As originally noted in \cite{Mar08a}, we leverage on the strong continuity of the nonlinearity in $L^2_{\rm loc}([0,\infty);\dot H^{-\frac12}(\T^2))$. However, since in general the embedding $L^\frac43(\T^2)\subset \dot H^{-\frac12}(\T^2)$ fails to be compact, the strong convergence in the Hamiltonian norm cannot follow from the uniform bound of $\{\theta^\nu\}_\nu$ in $L^\infty([0,\infty);L^\frac43(\T^2))$ alone. We overcome this issue by propagating\footnote{Here the transport-diffusion structure of \eqref{viscous-SQG} plays a fundamental role.} in time the equi-integrability of $\{|\theta^\nu(t)|^\frac43\}_\nu$, from which the strong $L^2_{\rm loc}([0,\infty);\dot H^{-\frac12}(\T^2))$ convergence is deduced via concentration compactness (see Proposition \ref{P:concentration compact}). 
 
Because of the plausible finite-time blow-up for \eqref{SQG}, such solutions cannot be obtained by regularizing the initial datum, and considering the viscous approximation \eqref{viscous-SQG} is necessary. In doing so, the conservation of the Hamiltonian becomes a non-trivial task in view of possible \quotes{anomalous dissipation} phenomena which may arise in the inviscid limit. This issue is handled in Theorem \ref{T:main vanish visc},  and more generally in Theorem \ref{T:compact implies no AD}, ruling out the dissipative anomaly in this context. In particular, we improve on \cite{CIN18}*{Remark 1.5} where the conservation of the Hamiltonian in the inviscid limit was obtained for $L^2(\T^2)$ initial data. These results have similarities with \cites{CLLS16,LMP21,DRP24,JLLL24,ELL_tocome,DeRosaPark2025} obtained in the context of the two-dimensional Navier--Stokes equations. Note that both Theorem \ref{T:main vanish visc} and Theorem \ref{T:compact implies no AD} only assume \quotes{Onsager supercritical} regularity, thus the absence of dissipation cannot follow by the Hamiltonian conservation of the limit which would instead require $\theta\in L^3([0,T]\times \T^2)$ (see \cite{IV15}).  Nonetheless, the Onsager conjecture for \eqref{SQG}, i.e. the existence of infinitely many non-conservative weak solutions, has been recently established in \cites{IsettLooi2024,DaiGiriRadu2024}. Earlier results were previously obtained in \cites{IV15,IsettMa2020,BSV19,BulutHuynhPalasek2023,DaiPeng2023}. See also \cite{CastroFaracoMengualSolera2025} for a non-uniqueness proof without convex integration. It follows that our Theorem \ref{T:compact implies no AD} establishes a sharp discrepancy between the situation in the inviscid case and vanishing viscosity limits in which the only strong compactness in the Hamiltonian norm is inconsistent with its dissipation. This is new in the SQG context and might be a common mechanism of all systems having (at least) two conserved quantities at two different levels of regularity, i.e. a \quotes{double cascade}.

\section{Notations and preliminaries}\label{S:tools}
In this section we set up the notation and recall some basic facts.
\subsection{Sobolev spaces} 
Let $f:\T^2\rightarrow \R$. Denoting by $\hat{f}(n)$ its Fourier coefficients we have
$$
f(x)=\sum_{n \in \mathbb{Z}^2} \hat{f}(n) e^{i n\cdot x}.
$$
For simplicity we will only consider functions with zero average. Therefore $\hat f(0)=0$ and we can sum over $n\in \mathbb{Z}^2\setminus\{0\}$ only. 

For any $s\in \R$ we define the homogeneous Sobolev norm by
$$
\|f\|_{\dot{H}^s}^2 := \sum_{n \in \mathbb{Z}^2\setminus\{0\}} |n|^{2s}|\hat{f}(n)|^2,
$$
which is induced by the inner product 
\begin{equation}
    \label{inner prod H - 12}
    \langle f , g \rangle_{\dot H^{s}}:=\sum_{n \in \mathbb{Z}^2\setminus\{0\}} |n|^{2s} \hat{f}(n) \hat{g}(n).
\end{equation}
For any $\alpha \in \R$ and any zero average function $f:\T^2\rightarrow \R$, the fractional Laplacian $(-\Delta)^\alpha f$ is defined as 
$$
(-\Delta)^\alpha f (x)=|\nabla |^{2\alpha }f(x):= \sum_{n \in \mathbb{Z}^2\setminus\{0\}}|n|^{2\alpha} \hat{f}(n) e^{i n\cdot x}.
$$
It follows
$$
\|f\|_{\dot{H}^s}=\|(-\Delta)^{\frac{s}{2}}f\|_{L^2}= \| |\nabla |^{s } f\|_{L^2} \qquad \forall s\in \R.
$$

\subsection{The C\'ordoba--C\'ordoba inequality}
The following is a direct consequence of the so-called C\'ordoba--C\'ordoba inequality. See for instance \cites{CC03,cordoba2004maximum} and \cite{CM15}*{Theorem 2.1}.

\begin{lemma}\label{lem.positivity} Let $\alpha \in (0,1]$ and $\beta : \mathbb{R} \to \mathbb{R}$ be a convex function, $\beta\in C^1(\R)$. For any $f \in \dot H^{2\alpha}(\T^2) \cap L^\infty(\T^2)$ there holds 
\begin{equation}
    \int_{\T^2} \beta'(f(x))(-\Delta)^{\alpha}f(x)\, dx \geq 0. 
\end{equation}
\end{lemma}
Note that the above integral is well-defined since $\beta'(f)\in L^\infty(\T^2)$ and $(-\Delta)^{\alpha}f\in L^2(\T^2)$ as soon as $f \in \dot H^{2\alpha}(\T^2)\cap L^\infty(\T^2)$. In fact, Lemma \ref{lem.positivity} will only be applied to the case when $f\in C^\infty(\T^2)$.
\subsection{Equi-integrability and concentration compactness}

We begin with the following quantitative $\dot H^{-\frac12}(\T^2)$ decay of the Fourier series' tail  in terms of the $L^\frac43(\T^2)$ integrability. 

\begin{proposition}
    \label{P: equi int quantitative}
    Let $f\in L^\frac43(\T^2)$. Assume there exists $\beta:\R_+\rightarrow \R_+$ such that $\lim_{r\rightarrow \infty}\frac{\beta(r)}{r}=\infty$ and 
    $$
\int_{\T^2} \beta(|f(x)|^\frac43)\,dx=:M<\infty.
    $$
Denote  $f_{>N}:=\sum_{|n|>N}\hat f(n) e^{i n\cdot x}$. For any $\eps>0$ there exists $R_\eps>0$ such that 
$$
\|f_{>N}\|^2_{\dot H^{-\frac12}}\leq C\left(\frac{R_\eps}{N}+\eps M^\frac32\right) \qquad \forall N>0,
$$
for some geometric constant $C>0$.
\end{proposition}
\begin{proof}
    Let $\eps>0$. Choose $R_\eps>0$ large enough such that 
    \begin{equation}
        \label{beta small}
        \frac{r}{\beta(r)}<\eps \qquad \forall r\geq R_\eps.
    \end{equation}
    Set $F_\eps:=\left\{x\in \T^2\, :\, |f(x)|\leq R_\eps\right\}$. Then $f=f\big|_{F_\eps} +f\big|_{F^c_\eps}=:f_{1,\eps} +f_{2,\eps}$. In particular $\hat f(n) = \hat f_{1,\eps}(n) +\hat f_{2,\eps}(n)$ for all $n\in \Z^2$. Thus
    \begin{align}
        \|f_{>N}\|^2_{\dot H^{-\frac12}}&=\sum_{|n|>N} |n|^{-1} |\hat f(n)|^2\\
        &\leq 2\left(\sum_{|n|>N} |n|^{-1} |\hat f_{1,\eps}(n)|^2  +\sum_{|n|>N} |n|^{-1} |\hat f_{2,\eps}(n)|^2\right)\\
        &\leq 2\left(\frac{\|f_{1,\eps}\|^2_{L^2}}{N}+\|f_{2,\eps}\|^2_{\ham}\right)\\
        &\leq 2\left(|\T^2|\frac{R_\eps^2}{N}+\|f_{2,\eps}\|^2_{L^\frac43}\right),
    \end{align}
    where in the last inequality we have used the point-wise bound $|f_{1,\eps}|\leq R_\eps$ and the Sobolev embedding $L^\frac43(\T^2)\subset \ham (\T^2)$. Moreover, by \eqref{beta small} we get
    $$
    \int_{\T^2} |f_{2,\eps}(x)|^\frac43\,dx= \int_{F^c_\eps} |f(x)|^\frac43\,dx= \int_{F^c_\eps} \frac{|f(x)|^\frac43}{\beta(|f(x)|^\frac43)}\beta(|f(x)|^\frac43)\,dx\leq \eps M,
    $$
    concluding the proof.
\end{proof}
The corresponding concentration compactness result directly follows.
\begin{proposition}
    \label{P:concentration compact}
    Let $\{f_j\}_j\subset L^\frac43(\T^2)$ be a bounded sequence such that $f_j\overset{*}{\rightharpoonup}f$ in $\ham (\T^2)$. If $\{|f_j|^\frac43\}_j\subset L^1(\T^2)$ is weakly compact, then $f_j\rightarrow f$ in $\ham (\T^2)$.
\end{proposition}
\begin{proof}
    We denote 
    \begin{equation}\label{cuts freq}
    f_{\leq N}(x) := \sum_{|n|\leq N} \hat f(n) e^{i n\cdot x} \qquad \text{and}\qquad f_{> N}(x) := \sum_{|n|> N} \hat f(n) e^{i n\cdot x}.
        \end{equation}
    For any $N>0$ we split 
    \begin{align}
        \|f_j-f\|^2_{\ham} &= \|(f_j-f)_{\leq N}\|^2_{\ham} + \|(f_j-f)_{>N}\|^2_{\ham}\\ 
        & \leq \|(f_j-f)_{\leq N}\|^2_{\ham} + 2 \left( \|(f_j)_{>N}\|^2_{\ham} + \|f_{>N}\|^2_{\ham}\right).
    \end{align}
    Since $\{|f_j|^\frac43\}_j\subset L^1(\T^2)$ is weakly compact, by the De la Vall\'ee Poussin criterion (see for instance \cite{K08}*{Theorem 6.19}) we find a function $\beta:\R_+\rightarrow \R_+$ such that $\lim_{r\rightarrow \infty}\frac{\beta(r)}{r}=\infty$ and 
    $$
    \sup_{j\geq 1} \int_{\T^2} \beta(|f_j(x)|^\frac43)\,dx<\infty.
    $$
    Let $\eps>0$. By Proposition \ref{P: equi int quantitative} we deduce 
    $$
     \|f_j-f\|^2_{\ham}\leq \|(f_j-f)_{\leq N}\|^2_{\ham} +C\left(\frac{R_\eps}{N}+\eps\right) + 2\|f_{>N}\|^2_{\ham}.
    $$
    Since $(f_j)_{\leq N}\rightarrow f_{\leq N}$ in $\ham (\T^2)$ for any fixed $N$, we obtain 
    $$
    \limsup_{j\rightarrow \infty} \|f_j-f\|^2_{\ham}\leq C\left(\frac{R_\eps}{N}+\eps\right) + 2\|f_{>N}\|^2_{\ham}.
    $$
    By letting $N\rightarrow \infty$ we deduce 
    $$
     \limsup_{j\rightarrow \infty} \|f_j-f\|^2_{\ham}\leq C\eps.
    $$
   The arbitrariness of $\eps>0$ concludes the proof.
\end{proof}

\subsection{Weak solutions to \eqref{SQG} and \eqref{viscous-SQG}} 
In order to treat the nonlinear term in the weak formulation with low regularity of $\theta$, let us first observe the following cancellation, which was already reported in \cites{Res95,Mar08a}. Let  $\varphi\in C_c^\infty(\mathbb{T}^2)$ be any test function. We have
 \begin{align}
 \int_{\mathbb{T}^2} \theta u\cdot \nabla \varphi \,dx  &= \int_{\mathbb{T}^2}  \mathcal{R}^\perp\theta \cdot \nabla \varphi (-\Delta)^{\frac{1}{2}}(-\Delta)^{-\frac{1}2}\theta\, dx \\
 & = \int_{\mathbb{T}^2} \mathcal{R}^\perp \theta \cdot [\nabla \varphi , (-\Delta)^{\frac{1}{2}}](-\Delta)^{-\frac{1}{2}}\theta \,dx \\
 &\quad +  \int_{\mathbb{T}^2} \mathcal{R}^\perp \theta \cdot (-\Delta)^{\frac{1}{2}} \left( \nabla \varphi (-\Delta)^{-\frac{1}{2}}\theta \right) dx,
 \end{align}
 where $[\cdot,\cdot]$ is the usual commutator symbol. The last term can be further computed as
 \begin{align}
 \int_{\mathbb{T}^2} \mathcal{R}^\perp \theta \cdot (-\Delta)^{\frac{1}{2}} \left( \nabla \varphi (-\Delta)^{-\frac{1}{2}}\theta \right) \,dx &= \int_{\mathbb{T}^2}(-\Delta)^{-\frac{1}{2}}\theta \nabla \varphi \cdot \nabla^\perp \theta \,dx\\
 & = -\int_{\mathbb{T}^2} \theta \nabla^\perp (-\Delta)^{-\frac{1}{2}}\theta \cdot \nabla \varphi  \,dx\\
  & = - \int_{\mathbb{T}^2} \theta u \cdot \nabla \varphi  \,dx.
 \end{align} 
 Thus, we arrive at
 \begin{align}
 \int_{\mathbb{T}^2} \theta u\cdot\nabla \varphi\, dx &= \frac{1}{2} \int_{\mathbb{T}^2} \mathcal{R}^\perp \theta \cdot [\nabla \varphi , (-\Delta)^{\frac{1}{2}}](-\Delta)^{-\frac{1}{2}}\theta\, dx\\
 &=:- \frac12 \left\langle \mathcal R^\perp_i \theta ,[(-\Delta)^{\mez},\partial_i \varphi  ](-\Delta)^{-\mez}\theta \right\rangle.\label{cant_smile_without_you}
 \end{align}
This expression for the nonlinear term motivates the following definition.

\begin{definition}[Weak solutions to \eqref{SQG}]\label{def.sqg-solutions} Let $\theta_0\in \dot H^{-\frac12}(\T^2)$ be with zero average. We say that $\theta \in L^{2}_{\rm loc}([0,\infty); \dot H^{-\frac12}(\T^2))$, with $\langle \theta(t),1\rangle=0$ for a.e. $t\geq 0$, is a weak solution to \eqref{SQG} if 
\begin{equation}
    \label{eq.sqg-weak-sol}
    \int_0^\infty \left\langle \theta (t),\partial_t\varphi (t)\right\rangle dt -\mez  \int_0^\infty \left\langle \mathcal R^\perp_i \theta (t),[(-\Delta)^{\mez},\partial_i \varphi (t) ](-\Delta)^{-\mez}\theta(t) \right\rangle dt   = -\left\langle\theta_0,\varphi(0)\right\rangle
\end{equation}
for all $\varphi \in C^{\infty}_c([0,\infty)\times \T^2)$, where $\mathcal R^\perp=\nabla^\perp(-\Delta)^{-\mez}$ and the brackets $\langle\cdot, \cdot\rangle$ denote the duality paring $ \dot H^{-\frac12}(\T^2)-\dot H^{\frac12}(\T^2)$.
\end{definition}

The above definition makes sense in view of the following.
\begin{remark}\label{R:nonlinearity well defined}
    As first noted in \cite{Mar08a} in the whole space setting, it can be proved \cite{BSV19}*{Lemma A.5} that the operator $[(-\Delta)^{\mez},\nabla \varphi]$ is of order zero and bounded on every $L^p(\T^2)$ and $\dot H^s(\T^2)$ for any $p\in (1,\infty)$ and any $s\in (-1,1)$. For instance, it holds
    \begin{align}
    \norm{ [(-\Delta)^{\mez},\nabla  \varphi  ] f}_{\dot{H}^{\frac{1}{2}}}&\le C\rVert \varphi\rVert_{C^3} \rVert f\rVert_{\dot{H}^{\frac{1}2}}.
    \end{align}
        In particular
   \begin{align}
\left|\left\langle \mathcal R^\perp_i \theta ,[(-\Delta)^{\mez},\partial_i \varphi  ](-\Delta)^{-\mez}\theta \right\rangle\right| &\le  C\rVert \varphi\rVert_{C^3} \norm{\mathcal{R}^\perp \theta}_{\dot{H}^{-\frac{1}{2}}}\norm{ (-\Delta)^{-\frac{1}2}\theta}_{\dot H^{\frac{1}2}}\\&\le C\rVert \varphi\rVert_{C^3} \rVert \theta\rVert_{\dot H^{-\frac{1}{2}}}^2,   \label{top_of_the_world}
\end{align}
   which ensures that the nonlinear term in \eqref{eq.sqg-weak-sol} is a well-defined continuous bilinear operator on $ L^2_{\rm loc}([0,\infty);\dot H^{-\frac{1}{2}}(\T^2))$. 
\end{remark}

We turn to the definition of weak solutions to \eqref{viscous-SQG}. 

\begin{definition}[Leray solutions to \eqref{viscous-SQG}]\label{def.viscous-sqg-sol}
Let $\nu>0$ and $\theta_0^\nu\in \dot H^{-\frac12}(\T^2)$ be with zero average. We say that $\theta^\nu$ is a Leray solution to \eqref{viscous-SQG} if 
\begin{itemize}
    \item[(i)] $\theta^\nu \in C^0([0,\infty);\dot H^{-\frac12}(\T^2))\cap L^2([0,\infty)\times \T^2)$ and 
    \begin{equation}\label{weak sol viscous}
    \int_0^\infty \int_{\T^2}\theta^\nu \left( \partial_t\varphi +\mathcal R^\perp \theta^\nu \cdot \nabla \varphi - \nu (-\Delta)^\frac12 \varphi\right)  dxdt= - \int_{\T^2}\theta_0^\nu(x)\varphi(x,0)\,dx
    \end{equation}
    for all $\varphi\in C^\infty_c([0,\infty)\times \T^2)$;
    \item[(ii)] it holds
    \begin{equation}
    \label{eq.viscous-conservation_propos}
    \|\theta^{\nu}(t)\|^2_{\dot{H}_x^{-\mez}} + 2\nu \int_0^t\|\theta^{\nu}(\tau)\|^2_{L^2_x} d\tau = \|\theta_0^\nu\|^2_{\dot{H}_x^{-\mez}} \qquad \forall t\geq 0;
\end{equation}
  \item[(iii)] it holds 
  \begin{equation}
    \label{eq.higher-bound}
    4\nu^{2}\int_0^\infty t\|\theta^{\nu}(t)\|^2_{\dot H^{\frac12}_x}dt \leq \|\theta_0^\nu\|^2_{\hamx}.
\end{equation}
\end{itemize}
\end{definition}
The reader may notice that the latter property \eqref{eq.higher-bound} is not commonly required in the definition of Leray solutions. However, it comes naturally as a uniform bound of any reasonable approximate sequence and it will be needed for our later analysis. By regularization and compactness arguments, Leray solutions exist. We provide the proof since we could not find the precise statement elsewhere. In fact, we are not aware of any result constructing Leray solutions satisfying the energy equality for a general initial datum with finite Hamiltonian. On the whole space, solutions satisfying the energy inequality were previously constructed in \cite{Mar08a} for any $\theta_0^\nu\in \dot H^{-\frac12}(\R^2)$. See also \cites{CI17,CIN18} for the bounded domain setting with $\theta_0^\nu\in L^2(\Omega)$. Uniqueness can be obtained under additional assumptions \cites{CW99,Mar08a,M08} but it is not known in general.

\begin{proposition}\label{P: Laray exists}
For any $\nu>0$ and $\theta_0^\nu\in \dot H^{-\frac12}(\T^2)$ with zero average there exists a Leray solution to \eqref{viscous-SQG} in the sense of Definition \ref{def.viscous-sqg-sol}.
\end{proposition}
\begin{proof}
 Since $\nu>0$ will be fixed, we will drop its notation from the superscripts. Consider $\theta_{0,\eps}:=\theta_0*\rho_\eps\in C^\infty(\T^2)$. By \cite{KNV07} we obtain a global-in-time smooth solution $\theta^\eps$ to \eqref{viscous-SQG} with initial datum $\theta_{0,\eps}$. Note that $\int_{\T^2}\theta^\eps(t,x)\,dx=0$ for all $t\geq 0$. By taking the $\dot H^{-\frac12}(\T^2)$ inner product (see \eqref{inner prod H - 12}) of the first equation in \eqref{viscous-SQG} and $\theta^\eps$ itself we obtain 
 $$
 \frac{d}{dt}\|\theta^\eps(t)\|^2_{\hamx}=-2\nu \|\theta^\eps(t)\|^2_{L^2_x}\qquad \forall t \geq 0,
 $$
 from which 
 $$
 \|\theta^\eps(t)\|^2_{\hamx}+2\nu \int_0^t \|\theta^\eps(\tau )\|^2_{L^2_x}\,d\tau = \|\theta_{0,\eps}\|^2_{\hamx}\qquad \forall t \geq 0.
 $$
 It follows that 
 \begin{equation}
     \label{first uniform bound}
     \{\theta^\eps\}_\eps\subset L^\infty([0,\infty);\dot H^{-\frac12}(\T^2))\cap L^2([0,\infty)\times \T^2)\qquad \text{is bounded.}
 \end{equation}
 Similarly, multiplying the first equation in \eqref{viscous-SQG} by $\theta^\eps$ we get 
 $$
 2\nu \int_s^t\|\theta^\eps(\tau)\|^2_{\dot H^{\frac12}_x}\,d\tau\leq \|\theta^\eps(t)\|^2_{L^2_x} + 2\nu \int_s^t\|\theta^\eps(\tau)\|^2_{\dot H^{\frac12}_x}\,d\tau=\|\theta^\eps(s)\|^2_{L^2_x}
 $$
 for all $t\geq s\geq 0$. Integrating for $s\in [0,t]$ we obtain 
 \begin{align}
 2\nu \int_0^t \tau \|\theta^\eps(\tau)\|^2_{\dot H^{\frac12}_x}\,d\tau&=2\nu \int_0^t\left(\int_s^t\|\theta^\eps(\tau)\|^2_{\dot H^{\frac12}_x}\,d\tau\right)ds\\
 &\leq \int_0^t \|\theta^\eps(s)\|^2_{L^2_x}\,ds\\
 &\leq \frac{1}{2\nu} \|\theta_{0,\eps}\|^2_{\hamx}\qquad \forall t\geq 0.
  \end{align}
a  Since $\|\theta_{0,\eps}\|_{\hamx}\leq \|\theta_{0}\|_{\hamx}$, this proves 
  \begin{equation}
      \label{getting higer order bound}
      4\nu^2 \int_0^\infty \tau \|\theta^\eps(\tau)\|^2_{\dot H^{\frac12}_x}\,d\tau\leq \|\theta_{0}\|_{\hamx}^2.
  \end{equation}
  Also 
  $$
  t \|\theta^\eps(t)\|^2_{L^2_x}\leq \int_0^t \|\theta^\eps(s)\|^2_{L^2_x}\,ds\leq \frac{1}{2\nu} \|\theta_0\|^2_{\hamx}\qquad \forall t>0.
  $$
  It follows that 
 \begin{equation}
     \label{second uniform bound}
     \{\theta^\eps\}_\eps\subset L^2_{\rm loc}((0,\infty);\dot H^{\frac12}(\T^2))\cap L^\infty_{\rm loc}((0,\infty);L^2(\T^2))\qquad \text{is bounded.}
 \end{equation}
In addition, by the structure of the nonlinearity (see Remark \ref{R:nonlinearity well defined}) we also get that
\begin{equation}
     \label{third uniform bound}
     \{\partial_t\theta^\eps\}_\eps\subset  L^\infty([0,\infty);\dot H^{-k}(\T^2))\qquad \text{is bounded}
 \end{equation}
for a sufficiently large $k\geq 1$. Thanks to \eqref{first uniform bound}, \eqref{second uniform bound} and \eqref{third uniform bound}, the Aubin--Lions lemma yields to a (non-relabeled) subsequence $ \{\theta^\eps\}_\eps$ and a scalar $\theta\in L^2([0,\infty)\times \T^2)$ such that 
\begin{equation}
    \label{strong in hamilt}
    \theta^\eps \rightarrow \theta \qquad \text{in } L^2_{\rm loc}([0,\infty);\dot H^{-\frac12}(\T^2))
\end{equation}
and
\begin{equation}
    \label{strong local time}
    \theta^\eps \rightarrow \theta \qquad \text{in } C^0_{\rm loc}((0,\infty);\dot H^{-\frac12}(\T^2))\cap L^2_{\rm loc}((0,\infty);L^2(\T^2)).
\end{equation}
Since 
$$
\int_0^\infty \int_{\T^2}\theta^\eps \left( \partial_t\varphi +\mathcal R^\perp \theta^\eps \cdot \nabla \varphi - \nu (-\Delta)^\frac12 \varphi\right)  dxdt= - \int_{\T^2}\theta_{0,\eps}(x)\varphi(x,0)\,dx
$$
for all $\varphi \in C^\infty_c([0,\infty)\times \T^2)$ and all $\eps>0$, by the strong convergence \eqref{strong in hamilt} we obtain\footnote{More precisely, the strong convergence \eqref{strong in hamilt} allows to pass to the limit in the nonlinearity when the latter is written as in \eqref{eq.sqg-weak-sol} (see Remark \ref{R:nonlinearity well defined}). Then, since $\theta\in L^2([0,\infty)\times \T^2)$, we obtain \eqref{weak sol viscous}.} \eqref{weak sol viscous}. Moreover, the estimate \eqref{eq.higher-bound} follows by \eqref{getting higer order bound} and the lower semicontinuity under weak convergence. We are left to show $\theta\in C^0([0,\infty);\dot H^{-\frac12}(\T^2))$ and the energy balance \eqref{eq.viscous-conservation_propos}.

By \eqref{strong local time} we can let $\eps\rightarrow 0$ in the energy balance for positive times to obtain 
\begin{equation}
    \label{admissibility}
    \|\theta(t)\|_{\hamx}\leq \|\theta_0\|_{\hamx}\qquad \forall t >0
\end{equation}
and 
\begin{equation}
    \label{balance for positive s}
    \|\theta(t)\|^2_{\dot{H}_x^{-\mez}} + 2\nu \int_s^t\|\theta(\tau)\|^2_{L^2_x} d\tau = \|\theta(s)\|^2_{\dot{H}_x^{-\mez}} \qquad \forall t\geq s> 0.
\end{equation}
Since $\theta\in C^0((0,\infty);\dot H^{-\frac12}(\T^2))\cap L^2([0,\infty)\times \T^2)$, a direct consequence of \eqref{weak sol viscous} is that $\theta(t)\overset{*}{\rightharpoonup} \theta_0$ in $\dot H^{-\frac12}(\T^2)$ as $t\rightarrow 0$. Together with \eqref{admissibility} this gives 
$$
\|\theta_0\|_{\hamx}\leq \liminf_{t\rightarrow 0}\|\theta(t)\|_{\hamx}\leq \limsup_{t\rightarrow 0}\|\theta(t)\|_{\hamx}\leq \|\theta_0\|_{\hamx}.
$$
In particular, the initial datum is achieved strongly in $\dot H^{-\frac12}(\T^2)$. This shows that  $\theta\in C^0([0,\infty);\dot H^{-\frac12}(\T^2))$. We can thus let $s\rightarrow 0$ in \eqref{balance for positive s} and conclude the validity of the energy balance \eqref{eq.viscous-conservation_propos}. 
\end{proof}

\section{Proof of the main theorems}

Theorem \ref{T:global-existence} is a direct consequence of Theorem \ref{T:main vanish visc}.

\begin{proof}[Proof of Theorem \ref{T:global-existence}]
    For any $\nu>0$ define $\eps_\nu:=\nu$ and let $\theta^\nu_0:=\theta_0*\rho_{\eps_\nu}$ be the mollification of $\theta_0\in L^\frac{4}{3} (\T^2)$. Let $\{\theta^\nu\}_\nu$ be the corresponding sequence of smooth solutions to \eqref{viscous-SQG}, which exists thanks to \cite{KNV07}. Since $\theta^\nu_0\rightarrow \theta_0$ in $L^\frac43(\T^2)$, we deduce $|\theta^\nu_0|^\frac43 \rightarrow |\theta_0|^\frac43$ in $L^1(\T^2)$. In particular $\{|\theta^\nu_0|^\frac43 \}_\nu\subset L^1 (\T^2)$ is weakly compact. We can thus apply Theorem \ref{T:main vanish visc} to obtain a scalar function $\theta$ satisfying all the requirements in the thesis of Theorem \ref{T:global-existence}.
\end{proof}

The proof of Theorem \ref{T:global-existence} we have presented here uses the fact that the spatial domain $\T^2$ has finite measure. This has been used in both Proposition \ref{P: equi int quantitative} and in the De la Vallée Poussin criterion invoked in the proof of Proposition \ref{P:concentration compact}.
\begin{remark}
    The weak solutions to \eqref{SQG} constructed in Theorem \ref{T:global-existence} could have also been obtained by any viscous approximation with subcritical dissipation  $(-\Delta)^\gamma$, $\gamma\in (\frac12,1]$. Indeed, the corresponding dissipative PDE is globally well-posed for smooth data. However, in this case, the conservation of the Hamiltonian, i.e. establishing the counterpart of Theorem \ref{T:main vanish visc}, becomes more involved. This happens because the dissipative term $(-\Delta)^\gamma$ breaks the natural scaling between the Hamiltonian and the higher-order bound \eqref{eq.higher-bound}. This issues has been recently addressed in \cite{DY26} which, in fact, generalizes the argument to the supercritical case $\gamma<\frac12$ as well. Note that global existence of smooth solutions is not known, and perhaps not expected \cite{K10}, in the supercritical regime. Therefore, for $\gamma<\frac12$, it becomes necessary to work with properly defined weak solutions already at the positive viscosity level.
\end{remark}

Before proving Theorem \ref{T:main vanish visc} we give the following general result ruling out anomalous dissipation. Since the sharp condition yielding to the conservation of the Hamiltonian for \eqref{SQG} is  $\theta \in L^3([0,T]\times \T^2)$, it must be noted that this is an Onsager supercritical condition.
\begin{theorem}
    \label{T:compact implies no AD}
    Let $\{\theta^\nu\}_{\nu}$ be a sequence of Leray solutions to \eqref{viscous-SQG} in the sense of Definition \ref{def.viscous-sqg-sol} with zero average initial data $\{\theta_0^{\nu}\}_\nu \subset \dot H^{-\frac12}(\T^2)$. Assume that $\{\theta_0^{\nu}\}_\nu$ is strongly compact in $\dot H^{-\frac12}(\T^2)$. Then 
    \begin{equation}\label{dissip small for small times}
       \forall \eps>0 \quad \exists \delta>0 \quad \text{s.t.} \quad\sup_{\nu\in (0,1)}\nu \int_0^\delta \|\theta^\nu(\tau)\|^2_{L^2_x}\, d\tau<\eps.
    \end{equation}
    Moreover, if in addition $ \{\theta^{\nu}\}_\nu\subset L^2([0,T];\dot H^{-\frac12}(\T^2))$ is strongly compact, then
    \begin{equation}
        \label{eq compact implies no AD}
    \lim_{\nu\rightarrow 0}\nu \int_0^T\|\theta^\nu(\tau)\|^2_{L^2_x}\, d\tau=0.
    \end{equation}
\end{theorem}
\begin{proof}
We prove the two claims separately.

\underline{\textsc{Proof of \eqref{dissip small for small times}}}. Let us denote 
\[
\Phi(N):= \sup_{\nu>0}\rVert \left(\theta_0^\nu\right)_{>N} \rVert_{\dot H^{-\frac{1}{2}}_x},
\]
where we are using the notation \eqref{cuts freq} for the frequency cutoff. Since $\left\{\theta_0^\nu\right\}_{\nu}$ is strongly compact in $\dot H^{-\frac{1}{2}}(\T^2)$, we have
\[
\lim_{N \to \infty}\Phi(N)=0.
\]
We compute
\begin{align}
\rVert \theta^\nu(t)-\theta^\nu_0\rVert_{\dot H^{-\frac{1}{2}}_x}^2&=\int_{\mathbb{T}^2}\left||\nabla|^{-\frac{1}{2}}\theta^\nu-|\nabla|^{-\frac{1}{2}}\theta^\nu_0\right|^2 dx\\
&=\rVert \theta^\nu(t)\rVert_{\dot H^{-\frac{1}2}_x}^2 - \rVert \theta_0^\nu\rVert_{\dot H^{-\frac{1}2}_x}^2 - 2\int_{\mathbb{T}^2} |\nabla|^{-\frac{1}{2}}\theta_0^\nu \left(|\nabla|^{-\frac{1}{2}}\theta^\nu-|\nabla|^{-\frac{1}{2}}\theta_0^\nu\right)dx\\
&\le - 2\int_{\mathbb{T}^2} |\nabla|^{-\frac{1}{2}}\theta_0^\nu \left(|\nabla|^{-\frac{1}{2}}\theta^\nu-|\nabla|^{-\frac{1}{2}}\theta_0^\nu\right)dx\\
&= - 2\int_{\mathbb{T}^2} |\nabla|^{-\frac{1}{2}}(\theta_0^\nu)_{>N} \left(|\nabla|^{-\frac{1}{2}}\theta^\nu-|\nabla|^{-\frac{1}{2}}\theta_0^\nu\right)dx\\
& \quad- 2\int_{\mathbb{T}^2} |\nabla|^{-\frac{1}{2}}(\theta_0^\nu)_{\leq N} \left(|\nabla|^{-\frac{1}{2}}\theta^\nu-|\nabla|^{-\frac{1}{2}}\theta_0^\nu\right)dx\\
&=: A -2 B,\label{Pianoman}
\end{align}
where the inequality follows from \eqref{eq.viscous-conservation_propos}. For $A$ we apply the Cauchy--Schwarz inequality to get
\begin{align}\label{Final_countdown}
|A|\le C\rVert (\theta^\nu_0)_{>N}\rVert_{\dot H^{-\frac{1}2}_x}\left(\rVert \theta^\nu(t)\rVert_{\dot H_x^{-\frac{1}{2}}}+\rVert \theta^\nu_0\rVert_{\dot H_x^{-\frac{1}{2}}}\right)\le C \Phi(N).
\end{align}
To estimate $B$, for any $ \varphi\in C^\infty(\mathbb{T}^2)$, we claim that
\begin{align}\label{Stayin_Alive_Beegees}
\sup_{\nu\in(0,1)}\left|\int_{\mathbb{T}^2} \varphi(x) (\theta^\nu(t,x)-\theta^\nu(t_0,x))\, dx\right|\le C\rVert \varphi\rVert_{C^3_x}|t-t_0|\qquad \forall 0\le t_0\le t,
\end{align}
for some constant $C>0$ that depends on $\sup_{\nu>0}\rVert \theta^\nu_0\rVert_{\dot H_x^{-\frac{1}{2}}}$ only.
Once the claim is proved, we plug in $t_0=0$ and $\varphi=(-\Delta)^{-1}(\theta_0^\nu)_{\leq N}$ and obtain
\[
\sup_{\nu\in (0,1)}|B|\le C \sup_{\nu>0}\rVert (-\Delta)^{-1}(\theta^\nu_0)_{\leq N}\rVert_{C^3_x}t \le CN^{10} t.
\]
Together with \eqref{Pianoman}   and \eqref{Final_countdown}, this proves 
\[
\rVert \theta^\nu(t)-\theta^\nu_0\rVert_{\dot H^{-\frac{1}{2}}_x}\le C\left( \Phi(N)+N^{10} t\right).
\]
 Therefore, for any $N>1$, the energy identity \eqref{eq.viscous-conservation_propos} can be estimated as
\begin{align}
\limsup_{t\rightarrow 0} \left( \sup_{\nu\in(0,1)}\nu\int_0^t \rVert \theta^\nu(\tau)\rVert_{L^2_x}^2 d\tau \right)&\le \limsup_{t\rightarrow 0} \sup_{\nu\in(0,1)}\left( \rVert \theta^\nu_0\rVert_{\dot{H}_x^{-\frac{1}{2}}}^2-\rVert \theta^\nu(t)\rVert_{\dot{H}_x^{-\frac{1}{2}}}^2\right)\\
&\le C\limsup_{t\rightarrow 0}\sup_{\nu\in(0,1)}\rVert \theta^\nu(t)- \theta^\nu_0\rVert_{\dot H^{-\frac{1}2}_x}\\
&\le C\Phi(N).
\end{align}
Since $N$ is arbitrary, we conclude  
\[
\limsup_{t\rightarrow 0} \left( \sup_{\nu\in(0,1)}\nu\int_0^t \rVert \theta^\nu(\tau)\rVert_{L^2_x}^2 d\tau \right)=0,
\]
which gives \eqref{dissip small for small times}.

To prove the claim \eqref{Stayin_Alive_Beegees} we use the weak formulation\footnote{Since $\theta^\nu\in C^0([0,\infty);\dot H^{-\frac12}(\T^2))$ we can let the test function in \eqref{weak sol viscous} converge to $\mathbbm{1}_{[t_0,t]}(\tau)\varphi(x)$.} \eqref{weak sol viscous} of Leray solutions
\begin{align}
 \int_{\mathbb{T}^2}\varphi(x)(\theta^\nu(t,x)-\theta^\nu(t_0,x))\,dx & = \int_{t_0}^t\int_{\mathbb{T}^2}  \theta^\nu\left(\mathcal{R}^\perp\theta^\nu\cdot \nabla \varphi -\nu(-\Delta)^{\frac{1}{2}}\varphi\right) dxd\tau \\
& =  \underbrace{\int_{t_0}^t\int_{\mathbb{T}^2}  \theta^\nu \mathcal{R}^\perp \theta^\nu\cdot \nabla \varphi  \,dxd\tau}_{=:B_1}\\
&\quad -\underbrace{\nu\int_{t_0}^t\int_{\mathbb{T}^2} |\nabla|^{-\frac{1}{2}}\theta^\nu(-\Delta)^{\frac{3}{4}}\varphi \,dxd\tau}_{=:B_2}.
\end{align}
Since $\rVert \theta^\nu(t)\rVert_{\dot H_x^{-\frac{1}{2}}}\leq \rVert \theta^\nu_0\rVert_{\dot H_x^{-\frac{1}{2}}}$, we estimate $B_2$ as
\begin{equation}\label{carpenters}
\sup_{\nu\in(0,1)}|B_2|\le C  |t-t_0|\rVert \varphi\rVert_{\dot H_x^{\frac{3}{2}}}\le  C  |t-t_0| \rVert \varphi\rVert_{C^3_x},
\end{equation}
where the constant $C>0$ depends only on $\sup_{\nu>0}\rVert \theta^\nu_0\rVert_{\dot H^{-\frac{1}2}_x}$.
For $B_1$, we use the identity \eqref{cant_smile_without_you} and the continuity estimate \eqref{top_of_the_world} to get
\begin{align}
|B_1|\le  C|t-t_0|\rVert \varphi\rVert_{C^3_x}\rVert \theta^\nu_0\rVert_{\dot H^{-\frac{1}2}_x}^2\le C|t-t_0|\rVert \varphi\rVert_{C^3_x}.
\end{align}
 We have proved \eqref{Stayin_Alive_Beegees}.

\underline{\textsc{Proof of \eqref{eq compact implies no AD}}}. Since  $\{\theta^{\nu}\}_\nu$ is assumed to be compact, we find a subsequence such that $\theta^\nu\rightarrow \theta$ in $L^2([0,T];\dot H^{-\frac12}(\T^2))$. Let $\eps>0$ and choose $\delta>0$ such that \eqref{dissip small for small times} holds. Then we split 
\begin{equation}
    \label{split diss small e large times}
   \nu \int_0^T\|\theta^\nu(\tau)\|^2_{L^2_x}\, d\tau<\eps + \nu \int_\delta^T\|\theta^\nu(\tau)\|^2_{L^2_x}\, d\tau.
\end{equation}
Moreover, by \eqref{eq.higher-bound} 
\begin{align}
    \nu \int_\delta^T\|\theta^\nu(\tau)\|^2_{L^2_x}\, d\tau& = \nu \int_\delta^T\int_{\T^2}|\nabla|^{-\frac12} \theta^\nu |\nabla|^{\frac12} \theta^\nu\, dxd\tau \\
    &=\nu \int_\delta^T\int_{\T^2}|\nabla|^{-\frac12} (\theta^\nu -\theta)|\nabla|^{\frac12} \theta^\nu\, dxd\tau \\
    &\quad + \nu \int_\delta^T\int_{\T^2}|\nabla|^{-\frac12} \theta |\nabla|^{\frac12} \theta^\nu\, dxd\tau\\
    &\leq \nu \left(\int_\delta^T \|(\theta^\nu-\theta)(\tau)\|^2_{\dot H^{-\frac12}_x}\,d\tau\right)^\frac12 \left(\int_\delta^T \|\theta^\nu(\tau)\|^2_{\dot H^{\frac12}_x}\,d\tau \right)^\frac12\\
    &\quad +\nu \int_\delta^T\int_{\T^2}|\nabla|^{-\frac12} \theta |\nabla|^{\frac12} \theta^\nu\, dxd\tau\\
    &\leq C_\delta \left(\int_\delta^T \|(\theta^\nu-\theta)(\tau)\|^2_{\dot H^{-\frac12}_x}\,d\tau\right)^\frac12 \\
    &\quad +\nu \int_\delta^T\int_{\T^2}|\nabla|^{-\frac12} \theta |\nabla|^{\frac12} \theta^\nu \, dxd\tau.
\end{align}
The first term vanishes since $\theta^\nu\rightarrow \theta$ in $L^2([0,T];\dot H^{-\frac12}(\T^2))$. We need to handle the second term. Note that $\{\nu |\nabla |^\frac12 \theta^\nu\}_\nu$ stays bounded in $L^2([\delta, T]\times \T^2)$ thanks to \eqref{eq.higher-bound}. Moreover, since $\{\theta^\nu\}_\nu$ is bounded in $L^\infty([0,\infty);\ham (\T^2))$, we also get  
$$
\nu \int_\delta^T\int_{\T^2}\varphi |\nabla|^{\frac12} \theta^\nu \, dxd\tau= \nu \int_\delta^T\int_{\T^2} |\nabla|^{-\frac12}\theta^\nu|\nabla| \varphi \, dxd\tau\rightarrow 0
$$
for all $\varphi\in C^\infty([0,T]\times \T^2)$. Thus $\nu |\nabla |^\frac12 \theta^\nu\rightharpoonup 0$ in $L^2([\delta, T]\times \T^2)$ and consequently 
$$
\lim_{\nu\rightarrow 0}\nu \int_\delta^T\int_{\T^2}|\nabla|^{-\frac12} \theta |\nabla|^{\frac12} \theta^\nu\, dxd\tau=0.
$$
This shows 
$$
\limsup_{\nu\rightarrow 0} \nu \int_\delta^T\|\theta^\nu(\tau)\|^2_{L^2_x}\, d\tau=0.
$$
Letting $\nu\rightarrow 0$ in \eqref{split diss small e large times} yields to 
$$
\limsup_{\nu\rightarrow 0} \nu \int_0^T\|\theta^\nu(\tau)\|^2_{L^2_x}\, d\tau<\eps.
$$
Since $\eps>0$ was arbitrary, we deduce that, along the subsequence we picked at the beginning, there is no anomalous dissipation. By repeating the above argument starting from any arbitrary subsequence, we obtain a further subsequence along which there is no dissipation. This proves \eqref{eq compact implies no AD}.
\end{proof}

We can now prove Theorem \ref{T:main vanish visc}. 

\begin{proof}[Proof of Theorem \ref{T:main vanish visc}]
        Note that, since the initial data are smooth, the solutions will be smooth for all times and all $\nu>0$ (see \cite{KNV07}). As always, subsequences will not be relabeled. By the $L^\frac43(\T^3)$ boundedness of the initial data
        , we find a $\theta_0\in L^\frac43(\T^3)$ such that $\theta^\nu_0\rightharpoonup \theta_0$ in $L^\frac43(\T^2)$ and $\theta^\nu_0\overset{*}{\rightharpoonup} \theta_0$ in $\dot H^{-\frac12}(\T^2)$. Since $\{|\theta^\nu_0|^\frac43\}_\nu\subset L^1(\T^2)$ is weakly compact, by Proposition \ref{P:concentration compact} we deduce 
    \begin{equation}\label{strong in hamilt initial data}
        \theta^\nu_0\rightarrow \theta_0 \qquad \text{in } \dot H^{-\frac12}(\T^2).
    \end{equation}
    Let $\beta\in C^1(\R)$ be convex. Multiplying the first equation in \eqref{viscous-SQG} by $\beta'(\theta^\nu)$ we get 
    $$
    \partial_t \beta(\theta^\nu)+ u^\nu\cdot \nabla \beta(\theta^\nu)= -\nu\beta'(\theta^\nu)(-\Delta)^\frac12\theta^\nu.
    $$
    Thanks to Lemma \ref{lem.positivity} this gives 
    \begin{equation}
        \label{renormalizing sqg visc}
        \int_{\T^2}  \beta(\theta^\nu (t,x))\,dx\leq  \int_{\T^2}  \beta(\theta_0^\nu (x))\,dx\qquad \forall t \geq 0.
    \end{equation}
    Choosing $\beta(r)=|r|^\frac43$, by the previous inequality we deduce that $\{\theta^\nu\}_\nu$ stays bounded in $L^\infty([0,\infty);L^\frac43(\T^2))$. Note that $\theta^\nu$  satisfies the energy balance \eqref{eq.viscous-conservation_propos} by smoothness. Thus
    \begin{equation}
        \label{theta nu bounded in time}
        \|\theta^\nu(t)\|_{\hamx}\leq  \|\theta^\nu_0\|_{\hamx}\qquad \forall t\geq 0.
    \end{equation}
    We obtain a subsequence $\theta^\nu\overset{*}{\rightharpoonup}\theta$  in $L^\infty([0,\infty);L^\frac43(\T^2))$ and in $L^\infty([0,\infty);\dot H^{-\frac12}(\T^2))$. In particular, by also using \eqref{strong in hamilt initial data}, we deduce
    \begin{equation}
        \label{lower semi ham}
         \|\theta(t)\|_{\hamx}\leq  \|\theta_0\|_{\hamx} \qquad \text{for a.e. } t\geq 0
    \end{equation}
    and 
      \begin{equation}
        \label{lower semi L43}
         \|\theta(t)\|_{L^\frac43_x}\leq  \liminf_{\nu\rightarrow 0}\|\theta^\nu_0\|_{L^\frac43_x} \qquad \text{for a.e. } t\geq 0
    \end{equation}
    by lower semicontinuity. We break the rest of the proof down into steps.

    \underline{\textsc{Step 1}}: $\theta \in C^0_{\rm w}([0,\infty);L^\frac43(\T^2))$.

    We will show that $\theta$, the weak* limit of $\theta^\nu$, can be redefined on a negligible set of times so that $\theta \in C^0_{\rm w}([0,\infty);L^\frac43(\T^2))$. 
     Since smooth functions are dense in $L^{4}(\T^2)$, it suffices to show that the sequence of functions 
     $$
     t\mapsto\int_{\T^2} \theta^\nu(t,x)\varphi(x)\,dx
     $$
     is equi-continuous for all $ \varphi \in C^\infty(\mathbb{T}^2)$, which was already proved in \eqref{Stayin_Alive_Beegees}. Indeed, by the  Ascoli--Arzelà theorem\footnote{The domain $[0,\infty)$ is locally compact and separable while the target space $L^\frac{4}{3}(\T^2)$ is compact when endowed with the weak topology.}, this yields the existence of an element $\tilde \theta \in C^0_{\rm w}([0,\infty);L^\frac43(\T^2))$ such that, possibly up to a further subsequence,
    \begin{equation}
        \label{weak conv Lp for all times}
    \theta^\nu(t)\rightharpoonup\tilde \theta(t)\qquad \text{in }  L^{\frac43}(\T^2),\, \forall t\geq 0.
    \end{equation}
    Then, $\tilde \theta(t)= \theta(t)$ for almost every $t\geq 0$ necessarily. Moreover, the two limits \eqref{strong in hamilt initial data} and \eqref{weak conv Lp for all times} together imply $\tilde \theta(0)=\theta_0$. From now on, we will work with the continuous in time, in the weak topology of $L^\frac43(\T^2)$, representative. By slightly abusing notation we will still denote it by $\theta$. 

    \underline{\textsc{Step 2}}: $\theta^\nu(t)\rightarrow \theta(t)$ in $\dot H^{-\frac12}(\T^2)$ for all $t\geq 0$.

    Since $H^\frac12(\T^2)\subset L^4(\T^2)$, by \eqref{weak conv Lp for all times} we deduce 
    \begin{equation}
        \label{weak conv for all times}
        \theta^\nu(t)\overset{*}{\rightharpoonup}\theta(t)\qquad \text{in } \dot H^{-\frac12}(\T^2),\, \forall t\geq 0.
    \end{equation}

We now wish to apply Proposition \ref{P:concentration compact}. Since $\{|\theta^\nu_0|^\frac43\}_\nu\subset L^1(\T^2)$ is weakly compact, by the De la Vall\'ee Poussin criterion (see for instance \cite{K08}*{Theorem 6.19}) we find a convex function\footnote{The criterion is usually stated without the requirement $\beta\in C^1(\R_+)$. However, this can easily be enforced via regularization.} $\beta\in C^1(\R_+)$ such that $\lim_{r\rightarrow \infty}\frac{\beta(r)}{r}=\infty$ and 
$$
\sup_{\nu>0} \int_{\T^2} \beta(|\theta^\nu_0(x)|^\frac43)\,dx<\infty.
$$
The function $\tilde\beta(r):=\beta(|r|^\frac43)$ is convex and continuously differentiable. We can thus apply \eqref{renormalizing sqg visc} with $\tilde\beta$ and deduce 
\begin{equation}\label{uniform in time equi integr}
\sup_{t,\nu>0} \int_{\T^2} \beta(|\theta^\nu(t,x)|^\frac43)\,dx<\infty.
\end{equation}
Then, the De la Vall\'ee Poussin criterion implies that $\{|\theta^\nu (t)|^\frac43\}_\nu\subset L^1(\T^2)$ is weakly compact for all $t\geq 0$. Therefore, by Proposition \ref{P:concentration compact} we conclude $\theta^\nu(t)\rightarrow \theta(t)$ in $\dot H^{-\frac12}(\T^2)$ for all $t\geq 0$.  This proves $(a)$. Note that, the point-wise in time strong convergence in $\dot H^{-\frac12}(\T^2)$, together with \eqref{theta nu bounded in time} and \eqref{lower semi ham}, yields to 
     \begin{equation}
         \label{theta is compact space time}
         \theta^\nu\rightarrow \theta \qquad \text{in } L^2_{\rm loc}([0,\infty);\dot H^{-\frac12}(\T^2))
     \end{equation}
     by the Lebesgue dominated convergence theorem.

     \underline{\textsc{Step 3}}: $\theta$ solves \eqref{SQG}.

     Clearly $\theta^\nu$ satisfies \eqref{weak sol viscous}. We can thus write the nonlinearity as in \eqref{eq.sqg-weak-sol} and, in view of Remark \ref{R:nonlinearity well defined} and \eqref{theta is compact space time}, deduce that $\theta$ is a weak solution to \eqref{SQG} in the sense of Definition \ref{def.sqg-solutions} with initial datum $\theta_0$. This proves $(b)$.

      \underline{\textsc{Step 4}}: Hamiltonian conservation.

      By \eqref{theta is compact space time} and Theorem \ref{T:compact implies no AD} we deduce the validity of $(c)$. In particular, by letting $\nu\rightarrow 0$ in \eqref{eq.viscous-conservation_propos} we deduce $\|\theta(t)\|_{\hamx}=\|\theta_0\|_{\hamx}$ for all $t\geq 0$. Moreover, the time continuity $\theta \in C^0_{\rm w}([0,\infty);L^\frac43(\T^2))$ proved in step 1 upgrades \eqref{lower semi L43} to hold for all $t\geq 0$ by the lower semicontinuity of the norm under weak convergence. Note that 
      $$
      \theta \in C^0_{\rm w}([0,\infty);L^\frac43(\T^2))\subset C^0_{\rm w}([0,\infty);\dot H^{-\frac12}(\T^2))
      $$
      which, together with the conservation of the Hamiltonian for all times,  yields to the strong continuity $\theta \in C^0([0,\infty);\dot H^{-\frac12}(\T^2)).$  The proof of $(d)$, and thus of the theorem, is concluded.
\end{proof}

\begin{remark}\label{R:uniform in time conv}
By the Ascoli--Arzelà theorem the weak convergence in \eqref{weak conv Lp for all times} is uniform in time on every compact subset of $[0,\infty)$. In particular, by Proposition \ref{P: equi int quantitative} applied to the time dependent sequence $\{\theta^\nu\}_\nu$, it holds $\theta^\nu\rightarrow \theta$ in $C^0_{\rm loc}([0,\infty);\dot H^{-\frac12}(\T^2))$.
\end{remark}

We conclude this section by giving an improved version of Theorem \ref{T:compact implies no AD}, identifying the relevant scales at which the strong compactness becomes effective on the dissipation.  For scaling reasons due to the dissipation $\nu(-\Delta)^\frac12$, this corresponds to length scales $\ell_\nu\sim \nu$. In terms of  frequencies this becomes $N_\nu\sim \nu^{-1}$. Compactness at these frequencies becomes equivalent to no anomalous dissipation of the Hamiltonian. The terminology ``compactness at a given frequency'' is motivated by the fact that a bounded sequence $\{f_j\}_j\subset \ham(\T^2)$ is strongly compact in $\ham(\T^2)$ if and only if 
$$
\lim_{N\rightarrow \infty} \sup_{j\geq 1}\sum_{|n|>N}|n|^{-1}|\hat f_j(n)|^2=0.
$$

\begin{theorem}
    \label{T: comp diss scale}
     Let $\{\theta^\nu\}_{\nu}$ be a sequence of Leray solutions to \eqref{viscous-SQG} in the sense of Definition \ref{def.viscous-sqg-sol} with zero average initial data, $\{\theta_0^{\nu}\}_\nu \subset \dot H^{-\frac12}(\T^2)$ strongly compact. Then 
     \begin{equation}\label{equivalence}
    \lim_{\nu\rightarrow 0}\sum_{|n|>\frac{c}{\nu}} |n|^{-1} \int_0^T \left| \widehat{\theta^\nu(\tau)}(n)\right|^2\,d\tau =0\quad \forall c>0\quad \Longleftrightarrow\quad \lim_{\nu\rightarrow 0} \nu \int_0^T\|\theta^\nu(\tau)\|^2_{L^2_x}\, d\tau=0.
       \end{equation}
\end{theorem}
\begin{proof}
   For any $\eps>0$ choose $\delta>0$ such that \eqref{dissip small for small times} holds and bound 
   \begin{equation}
       \label{usual split dissipation}
       \nu \int_0^T\|\theta^\nu(\tau)\|^2_{L^2_x}\, d\tau<\eps + \nu \int_\delta^T\|\theta^\nu(\tau)\|^2_{L^2_x}\, d\tau.
   \end{equation}
   Let $c>0$ and set $N_\nu:=c\nu^{-1}$. Denote by $\theta^\nu_{\leq N_\nu}$ and $\theta^\nu_{> N_\nu}$ the frequencies cuts as in \eqref{cuts freq}. We split
   \begin{equation}
       \label{split diss in cuts}
       \nu \int_\delta^T\|\theta^\nu(\tau)\|^2_{L^2_x}\, d\tau=\nu \int_\delta^T\|\theta_{\leq N_\nu}^\nu(\tau)\|^2_{L^2_x}\, d\tau + \nu \int_\delta^T\|\theta_{> N_\nu}^\nu(\tau)\|^2_{L^2_x}\, d\tau.
   \end{equation}
   Since 
   $$
   \|\theta_{> N_\nu}^\nu(\tau)\|^2_{L^2_x}\leq \|\theta_{> N_\nu}^\nu(\tau)\|_{\hamx} \|\theta_{> N_\nu}^\nu(\tau)\|_{\dot H^{\frac12}_x}\leq \|\theta_{> N_\nu}^\nu(\tau)\|_{\hamx} \|\theta^\nu(\tau)\|_{\dot H^{\frac12}_x},
   $$
   by \eqref{eq.higher-bound} we estimate 
   \begin{align}
       \nu \int_\delta^T\|\theta_{> N_\nu}^\nu(\tau)\|^2_{L^2_x}\, d\tau&\leq \nu \left(\int_\delta^T\|\theta_{> N_\nu}^\nu(\tau)\|^2_{\ham_x}\, d\tau\right)^\frac12 \left(\int_\delta^T\|\theta^\nu(\tau)\|^2_{\dot H^{\frac12}_x}\, d\tau\right)^\frac12 \\
       &\leq C_\delta \left( \sum_{|n|>N_\nu} |n|^{-1} \int_0^T \left| \widehat{\theta^\nu(\tau)}(n)\right|^2\,d\tau\right)^\frac12.\label{bound high freq}
   \end{align}
   Also, by \eqref{eq.viscous-conservation_propos}  
   \begin{equation}
       \label{bound low freq}
       \|\theta_{\leq N_\nu}^\nu(\tau)\|^2_{L^2_x} \leq N_\nu \|\theta^\nu(\tau)\|^2_{\hamx}\leq CN_\nu.
   \end{equation}
   Recalling that $N_\nu=c\nu^{-1}$, by plugging \eqref{bound high freq} and \eqref{bound low freq} into \eqref{split diss in cuts} we obtain 
   $$
    \nu \int_\delta^T\|\theta^\nu(\tau)\|^2_{L^2_x}\, d\tau\leq C_\delta \left(c + \left( \sum_{|n|>N_\nu} |n|^{-1} \int_0^T \left| \widehat{\theta^\nu(\tau)}(n)\right|^2\,d\tau\right)^\frac12\right).
   $$
   Letting $\nu\rightarrow 0$ yields to 
   $$
   \limsup_{\nu\rightarrow 0}\nu \int_\delta^T\|\theta^\nu(\tau)\|^2_{L^2_x}\, d\tau\leq C_\delta c.
   $$
   Going back to \eqref{usual split dissipation}, we have proved 
   $$
    \limsup_{\nu\rightarrow 0}\nu \int_0^T\|\theta^\nu(\tau)\|^2_{L^2_x}\, d\tau\leq \eps + C_\delta c,
   $$
   which then must vanish by first letting $c\rightarrow 0$ and then $\eps\rightarrow 0$. This proves the left-to-right implication. The reverse implication is a direct consequence of 
   $$
    \sum_{|n|>N} |n|^{-1} \int_0^T \left| \widehat{\theta^\nu(\tau)}(n)\right|^2\,d\tau\leq \frac{1}{N}  \int_0^T\|\theta^\nu(\tau)\|^2_{L^2_x}\, d\tau \qquad \forall N,\nu>0.
   $$
\end{proof}

\section{Comments and remarks}
In this final section we make some additional remarks about the results we have proved.

\subsection{Sharp explicit  rates for $p>\frac43$} When the initial data $\{\theta^\nu_0\}_\nu$ belong to a bounded subset of $L^p(\T^2)$ for some $p>\frac43$, the dissipation can be proved to vanish with an explicit algebraic rate. Furthermore, we provide an example showing the sharpness of such rate. 

To obtain an explicit rate we follow the strategy \quotes{Sobolev interpolation \& superquadratic Gr\"onwall} introduced in \cite{CLLS16}, where the first supercritical energy conservation result has been proved in an inviscid limit context.  We start by 
\begin{equation}\label{preparing gronwal}
\frac{d}{dt}\|\theta^\nu(t)\|^2_{L^2_x} = -2\nu \| \theta^\nu(t)\|^2_{\dot H^\frac12_x}.
\end{equation}
Note that when $p=2$ the bound 
\begin{equation}\label{trivial rate p 2}
\nu\int_0^T\|\theta^\nu(t)\|^2_{L^2_x}\,dt\leq \nu T\|\theta^\nu_0\|^2_{L^2_x}
\end{equation}
is trivial, providing a rate linear in $\nu$. In general, such bound cannot be improved\footnote{The function $\theta^{\nu}(x_1,x_2,t)=e^{-\nu t}\sin x_1$ solves \eqref{viscous-SQG} with initial datum $\theta_0(x_1,x_2)=\sin x_1$. Its dissipation is exactly of order $\sim \nu$.} for $p>2$.  Therefore we can  assume $\frac43\leq p<2$.
By interpolation we have
$$
\|\theta^\nu(t)\|_{L^2_x}\leq C\|\theta^\nu(t)\|_{L^p_x}^\alpha \| \theta^\nu(t)\|_{\dot H^\frac12_x}^{1-\alpha}\qquad \text{with } \alpha= \frac{p}{4-p}.
$$
Moreover, assuming $\{\theta^\nu_0\}_\nu\subset L^p(\T^2)$ to be bounded, we get 
$$
\sup_{\nu,t>0}\|\theta^\nu(t)\|_{L^p_x}\leq \sup_{\nu>0}\|\theta^\nu_0\|_{L^p_x}<\infty.
$$
In particular, the interpolation inequality above becomes 
$$
\|\theta^\nu(t)\|^{\frac{1}{1-\alpha}}_{L^2_x}\leq C\| \theta^\nu(t)\|_{\dot H^\frac12_x}\qquad \text{with } \alpha= \frac{p}{4-p}.
$$
By plugging it into \eqref{preparing gronwal} we achieve 
$$
\frac{d}{dt}\|\theta^\nu(t)\|^2_{L^2_x}\leq -C \nu\|\theta^\nu(t)\|^{\frac{2}{1-\alpha}}_{L^2_x}.
$$
Then, the Gr\"onwall inequality yields to 
$$
\| \theta^\nu (t)\|^2_{L^2_x}\leq \frac{C}{\left(\frac{\alpha}{1-\alpha}\right)^\frac{1-\alpha}{\alpha}}\frac{1}{(\nu t)^\frac{1-\alpha}{\alpha}},
$$
for a constant $C>0$ independent on $\alpha,\nu$ and $t$. If $\alpha>\frac12$, which corresponds to $p>\frac43$, the above bound is integrable  around $t=0$. Therefore,  we obtain
\begin{align}
\nu\int_0^T \| \theta^\nu (t)\|^2_{L^2_x}\,dt&\leq \frac{C}{\left(\frac{\alpha}{1-\alpha}\right)^\frac{1-\alpha}{\alpha}}\frac{\alpha}{2\alpha -1}\nu^{\frac{2\alpha -1}{\alpha}}\\
&=\frac{C}{\left(\frac{p}{4-2p}\right)^\frac{4-2p}{p}}\frac{p}{3p-4}\nu^{\frac{3p-4}{p}}\qquad \text{for } p\in\left(\frac43,2\right),\label{final rate}
\end{align}
that is the desired rate. The above constant $C>0$ does not depend on $p$ and $\nu$. Therefore, when $p\rightarrow 2$, or equivalently when $\alpha\rightarrow 1$, the above bound converges to $\nu$, matching the trivial rate \eqref{trivial rate p 2} discussed above. It is clear that when $p=\frac43$ nothing can be deduced by the above argument. It is then necessary to follow a different strategy, for instance the one used to prove Theorem \ref{T:main vanish visc}.

The rate \eqref{final rate} can be proved to be sharp. For convenience we will work on the whole space $\R^2$. Fix a non-zero radial $\theta_0\in C^\infty_c(\R^2)$ with zero average and solve 
$$
 \begin{array}{rcll}
  \partial_t\theta +(-\Delta)^\frac12 \theta  &=& 0 & \\ 
  \theta(0,\cdot)&=&\theta_0&
  \end{array}
$$
on $\R^2\times [0,\infty)$. For any $\nu>0$ and $p\geq 1$ set
$$
\theta^\nu_0(x):=\frac{1}{\nu^\frac{2}{p}}\theta_0\left(\frac{x}{\nu}\right).
$$
It follows that the function
$$
\theta^\nu(t,x):=\frac{1}{\nu^\frac{2}{p}}\theta\left(t,\frac{x}{\nu}\right)
$$
solves
$$
 \begin{array}{rcll}
  \partial_t\theta^\nu +\nu(-\Delta)^\frac12 \theta^\nu &=& 0 & \\ 
  \theta^\nu(0,\cdot)&=&\theta^\nu_0.&
  \end{array}
$$
By the radial symmetry, we have obtained a sequence of solutions $\{\theta^\nu\}_\nu$ to \eqref{viscous-SQG} with initial data bounded in $L^p(\R^2)$. Moreover, a direct computation shows 
$$
\nu \int_0^T\|\theta^\nu(t)\|^2_{L^2_x}\,dt= \nu^\frac{3p-4}{p}\int_0^T \|\theta(t)\|^2_{L^2_x}\,dt,
$$
proving the sharpness of the rate obtained in \eqref{final rate} for $\frac43<p\leq 2$.

\subsection{Necessity of the initial equi-integrability} 
Without the equi-integrability assumption on the initial data, Theorem \ref{T:main vanish visc} fails. Indeed, on the whole space, the example constructed above for $p=\frac{4}{3}$ provides a sequence $\{\theta^\nu\}_\nu$ solving \eqref{viscous-SQG} such that 
$$
\nu \int_0^T\|\theta^\nu(t)\|^2_{L^2_x}\,dt=\int_0^T \|\theta(t)\|^2_{L^2_x}\,dt>0.
$$
Note that the initial data $\{|\theta^\nu_0|^\frac43\}_\nu$ concentrate at the origin. Moreover, in view of Theorem \ref{T:compact implies no AD}, the sequence of solutions must loose the space-time compactness in the Hamiltonian norm. Indeed, for scaling reasons, the spatial atomic concentration appear in the sequence $\{ ||\nabla|^{-\frac12}\theta^\nu(t)|^2\}_\nu$ as well, uniformly in $t\in[0,T]$. Of course, here the pathological behavior of the solutions is propagated by the one in the initial data. Whether, given \quotes{good} initial data, a similar mechanism can be induced by the wild dynamics in time is widely open.

\subsection{Uniform in time concentration compactness} The key step to prove Theorem \ref{T:global-existence} is to establish the strong compactness of the vanishing viscosity sequence $\{\theta^\nu\}_\nu$ in $L^2_{\rm loc}([0,\infty);\dot H^{-\frac12}(\T^2))$. This guarantees indeed that the limit is a weak solution to \eqref{SQG} and, thanks to Theorem \ref{T:compact implies no AD}, that there is no anomalous dissipation, from which the conservation of the Hamiltonian in the inviscid limit is deduced. The strong compactness of the vanishing viscosity limit has been deduced by Proposition \ref{P:concentration compact} since $\{|\theta^\nu(t)|^\frac43\}_\nu$ is weakly compact, or analogously equi-integrable in space, uniformly in time. Such property is propagated by that on the initial data (see \eqref{uniform in time equi integr}).

Since the embedding $L^\frac{4}{3}(\T^2)\subset \dot H^{-\frac12}(\T^2)$ is continuous, the absence of atomic concentrations in the higher-order norm suffices to obtain strong compactness in the lower order one. This is indeed a consequence of the concentration compactness principle by Lions \cites{PLL1,PLL2}. However, as opposite to the equi-integrability of the initial data when raised to the right power\footnote{This guarantees that any weak* limit of the sequence $\{|\theta^\nu_0|^\frac43\}_\nu$ is absolutely continuous with respect to the Lebesgue measure.}, there is no currently  known mechanism which allows to propagate in time the non-atomic condition. This issue is related to the global existence of finite energy weak solutions to the two-dimensional incompressible Euler equations with measure initial vorticity \cites{delort1991existence,scho95,DM88,evans1994hardy,vecchi19931}. In that context, the largest class of initial vorticities for which the global existence is known is that of Delort \cite{delort1991existence}, i.e. measures with positive singular part. Indeed, having a sign restriction on the singular part allows to propagate the non-atomic condition in time. However, the lack of the  embedding $L^1(\T^2)\subset H^{-1}(\T^2)$ makes the implementation of the concentration compactness argument highly non-trivial: only special nonlinear combinations pass to the limit which, remarkably,  suffice to obtain a weak solution thanks to the  structure of the nonlinearity in the PDE. In the incompressible Euler equations the strong compactness in $H^{-1}(\T^2)$ might still be, in principle, lost, even if the initial vorticities would be equi-integrable. This is different to what happens for \eqref{SQG}, in which the validity of the continuous embedding between the two corresponding norms always allows to deduce the strong compactness in the lower order one as soon as the higher-order one does not display concentrations.

\bibliographystyle{plain} 
\bibliography{biblio}

\end{document}